\documentclass[a4paper]{amsart}

\usepackage[english]{babel}
\usepackage{graphicx}        
   \usepackage{url}                          

\makeindex             

\newcommand{\C}{\mathbb{C}}
\newcommand{\R}{\mathbb{R}}

\newcommand{\N}{\mathbb{N}}

\newcommand{\de}{\partial}
\newcommand{\dbar}{\overline{\partial}}
\newcommand{\debar}{\overline{\partial}}

\newcommand{\levinull}{\mathcal{N}}
\newcommand{\distr}{\mathcal{D}}
\newcommand{\core}{\mathfrak{C}}

\newcommand{\OO}{\mathcal{O}}
\newcommand{\ideal}{\mathcal{I}}
\newcommand{\alg}{\mathcal{A}}
\newcommand{\smooth}{\mathcal{C}}
\newcommand{\modR}{\mathcal{R}}

\newtheorem{theorem}{Theorem}[section]
\newtheorem{proposition}[theorem]{Proposition}
\newtheorem{lemma}[theorem]{Lemma}

\newtheorem{definition}[theorem]{Definition}
\newtheorem{ex}[theorem]{Example}
\newtheorem{remark}[theorem]{Remark}

	\title{Remarks on the Levi core}

\author{Gian Maria Dall'Ara}

	\address{Istituto Nazionale di Alta Matematica ``F. Severi" \\ Research Unit Scuola Normale Superiore \\ Piazza dei Cavalieri, 7, 56126, Pisa (Italy)} 
	\email{dallara@altamatematica.it}
\author{Samuele Mongodi}
\address{Dipartimento di Matematica  e Applicazioni\\ Università degli Studi di Milano-Bicocca\\ Via Roberto Cozzi, 55, I-20125, Milano (Italy)} \email{samuele.mongodi@unimib.it}

\begin{document}

\begin{abstract}We investigate a few aspects of the notion of Levi core, recently introduced by the authors in \cite{Levicore}: a basic finiteness question, the connection with Kohn's algorithm, and with Catlin's property (P). \end{abstract}

	\maketitle

	\section{Introduction}
	
This note complements the recent paper "The core of the Levi distribution" \cite{Levicore}, where the authors introduced a new geometric invariant associated to CR manifolds of hypersurface type, called the Levi core. We refer to the paper for motivations and applications to the regularity theory of the $\bar\partial$-Neumann problem. Here we want to address certain basic questions that arise quite naturally when considering the Levi core, and indeed have been raised by several people in private communications to the authors. 

The Levi core is a special case of a more general construction that attaches to any distribution $\mathcal{D}$ of subspaces of the tangent bundle of a manifold a smaller distribution $\mathfrak{C}(\mathcal{D})$, its core. The core is defined starting with $\mathcal{D}$ and iterating a "derived distribution" construction until it eventually stabilizes. This works in great generality, namely for any smooth manifold and any distribution, under no regularity assumption beyond the basic requirement that $\mathcal{D}$ be a closed subset of the tangent bundle. The price to be paid for that generality is that the stabilization may require a (countably) infinite number of iterations, a fact that is in the nature of things, as the notion of derived distribution collapses in the one-dimensional case to that of derived set (and this is indeed the origin of the terminology) in set theory. We refer again to \cite{Levicore} for a more detailed discussion. 

The transfinite nature of the core makes it difficult to extract some usable geometric structure from it. Since such a structure would be highly desirable in applications, for example in the case of the Levi core in connection with the $\bar\partial$-Neumann problem, a couple of basic natural questions arise: is the core $\mathfrak{C}(\mathcal{D})$ "reached" after finitely many iterations, under the assumption that the underlying manifold and the distribution $\mathcal{D}$ lie in an appropriate category smaller than $\mathcal{C}^\infty$? In that case, do the core and its support lie in the same category? After reviewing the terminology and definitions from \cite{Levicore} in Section 2, in Section 3 we observe that this is indeed the case in the complex algebraic and complex analytic categories, and also in the real analytic category under an additional coherence assumption. We also point out the difficulties that present themselves if one drops the coherence hypothesis, or ventures beyond the analytic category. 

The algebraic outlook of Section 3 leads then to the consideration of the relation between the notion of Levi core (that is, the core of the Levi null distribution of a CR manifold) and the well-known Kohn's algorithm. In Section 4, we  discuss this relation, the theme being that the two are indeed manifestations of the same idea, at least when restricting to nice enough categories. 

In Section 5, we discuss a couple of specializations of the notion of core to complex analysis, other than the Levi core. 

Finally, in Section 6 we show that there are bounded smooth pseudoconvex domains satisfying Catlin's property (P) and having nontrivial Levi core. This theme has been explored in more depth by Treuer \cite{Treuer}. 

\subsection*{Acknowledgments:} The authors would like to thank François Berteloot for drawing their attention on the notion of B-regularity in relation with the core and Francesca Acquistapace and Fabrizio Broglia for fruitful discussions on analytic and Carlemann functions.

	\section{Basic definitions}
	
	Let $M$ be a smooth manifold. We begin by recalling some terminology and definitions from \cite{Levicore}.
	
	A \textbf{real distribution} on $M$ is a subset $\distr$ of the tangent bundle $TM$ such that the fiber $\distr_p:=\distr\cap T_pM$ is a vector subspace of $T_pM$ for every $p\in M$. A \textbf{complex distribution} on $M$ is a subset $\distr$ of the complexified tangent bundle $\C TM$ such that $\distr_p:=\distr\cap \C T_pM$ is a complex vector subspace of $\C T_pM$ for every $p\in M$.
	
	The \textbf{support} $S_{\distr}$ of a distribution $\distr$ is the set of points $p\in M$ such that $\distr_p\neq \{0\}$. We call a distribution \textbf{closed} if it is a closed subset of $TM$ (or $\C TM$); for a closed distribution the function $p\mapsto \dim\distr_p$ is upper semicontinuous, and in particular the support $S_{\distr}$ is closed.
	
	Closed distributions naturally occur as common kernels of subsets 	$\mathfrak{R}$ of $\Omega^1(M)$, i.e. of smooth differential $1$-forms. More precisely, given such an $\mathfrak{R}$ one defines \[\distr_p=\{X_p\in T_pM\ :\ \langle \omega_{p}, X_p\rangle=0\quad \forall \omega\in \mathfrak{R} \},\]
	and $\distr=\cup_p\distr_p=:\ker \mathfrak{R}$. There is no loss of generality in assuming that $\mathfrak{R}$ is a $C^\infty(M)$-submodule of $\Omega^1(M)$. 
	
	As a special case of the above construction, given a set $S\subseteq M$, consider the associated ideal of smooth functions 
	\[\mathcal{I}(S)=\{f\in\mathcal{C}^\infty(M)\ :\ f\vert_S\equiv 0\}\] and the $C^\infty(M)$-module of $1$-forms $\mathfrak{R}(S)$ generated by $\{df\, :\ f\in\mathcal{I}(S)\}$. We call the associated distribution $\ker \mathfrak{R}(S)$ the \textbf{tangent distribution} to $S$, and we denote it by $TS$ (cf. \cite[Def. 2.5]{Levicore}). Notice that, according to this definition, $TS$ is a subset of $TM$ with zero fiber at every point outside the closure of $S$ (and also at every isolated point of $S$). 
	
	We now recall from \cite{Levicore} the two definitions that are most important for our discussion. 
	
	\begin{definition}[\textbf{Derived distribution}]\label{derived_distribution} Let $M$ be a smooth manifold and let $\distr\subseteq TM$ be a real distribution. The distribution \begin{equation*}
			\distr':=\distr\cap TS_{\distr}
		\end{equation*} will be called the \textbf{derived distribution} of $\distr$. Analogously, if $\distr\subseteq \C TM$ is a complex distribution on $M$, its derived distribution is defined as \begin{equation*}
			\distr':=\distr\cap \C TS_{\distr}.
		\end{equation*}
		Here $\C TS_{\distr}$ is obtained complexifying pointwise $TS_{\distr}$. 
	\end{definition}

A distribution that equals its derived distribution is said to be \textbf{perfect}. 
	
	It is possible to iterate the operation of "taking the derived distribution". Given an ordinal $\alpha$, we define $\distr^{(\alpha)}$ by transfinite recursion (see \cite{Ciesielski}, Theorem 4.3.1):\begin{enumerate}
		\item if $\alpha=0$, we set $\distr^{(0)}:=\distr$;
		\item if $\alpha+1$ is a successor ordinal, we set $\distr^{(\alpha+1)}:=(\distr^{(\alpha)})'$;
		\item if $\alpha$ is a limit ordinal, we set $\distr^{(\alpha)}:=\bigcap_{\beta<\alpha}\distr^{(\beta)}$.
	\end{enumerate}
	If $\distr$ is closed, then every $\distr^{(\alpha)}$ is closed (this is easily proved by transfinite induction) and we have the following Cantor--Bendixson-type theorem (cf. \cite[Thm. 2.9]{Levicore}). 
	
	\begin{theorem}\label{Cantor_Bendixson}
		If $\distr$ is a closed (real or complex) distribution, then there exists a countable ordinal $\alpha$ such that $\distr^{(\alpha)}$ is perfect.
	\end{theorem}

Thanks to Theorem \ref{Cantor_Bendixson}, we can formulate our second key definition. 

	\begin{definition}[\textbf{Core of a distribution}]\label{core_dfn}
		Let $\distr$ be a (real or complex) closed distribution. Then the core of $\distr$ is the distribution $\core(\distr)=\distr^{(\alpha)}$, where $\alpha$ is the minimal ordinal such that $(\distr^{(\alpha)})'=\distr^{(\alpha)}$.
	\end{definition}

Notice that both Definition \ref{derived_distribution} and Definition \ref{core_dfn} are local in nature: distributions can be restricted to open sets in the obvious way \[\distr_{\lvert U}=\cup_{p\in U}\distr_p\qquad \forall U\subseteq M\text{ open},\]and then $\mathcal{D}'_{|U}=(\mathcal{D}_{|U})'$ and $\core(\distr_{|U})=\core(\distr)_{|U}$.
	
	\section{The core in the algebraic and analytic categories}\label{sec:algebraic}
	
	In the construction of the core, we require that the distribution $\distr$ be closed; this property carries on to the derived distribution, because the support of a closed distribution is a closed set and the tangent distribution to a closed set is a closed distribution, and it is a key ingredient in proving that we need countably many steps to reach the core, which, in turn, is again a closed set.
	
	An analogous situation happens in the category of complex algebraic or analytic varieties.
	
	\begin{proposition}\label{prp_holalgcore}Let $X$ be a complex algebraic (or analytic) manifold and let $\distr\subseteq TX$ be a distribution which is also an algebraic (or analytic) subvariety of $TX$; then:
	\begin{enumerate}
	\item  $S_\distr$ is an algebraic (or analytic) subvariety of $X$, 
	\item $\distr'$ is  an algebraic (or analytic) subvariety of $TX$,
	\item $\core_{\distr}$ is locally reached in a finite number of steps, hence it is an algebraic (or analytic) subvariety of $TX$ and its support is an algebraic (or analytic) subvariety of $X$.
	\end{enumerate}
	\end{proposition}
	\begin{proof}
 We prove the three statements in order.
	
	\paragraph{$S_{\distr}$ is an algebraic (or analytic) subvariety of $X$.} $X\times\{0\}$ is a subvariety of $TX$ which is isomorphic to $X$ and it is contained in $\distr$. If $S_{\distr}=X$, then $S_{\distr}$ is trivially a subvariety of $X$, otherwise there are points $(x,0)\in TX$ where, locally, $X\times\{0\}$ and $\distr$ coincide (i.e. they coincide in an open neighborhood of $(x,0)\in TX$); therefore, $X\times\{0\}$ is an irreducible component of $\distr$ (by the decomposition in irreducible components of complex algebraic, or analytic, subvarieties). 
	Let us denote by $Y$ the union of all the other irreducible components of $\distr$, which is again a subvariety; then, via the isomorphism between $X$ and $X\times\{0\}\subseteq TX$, $S_{\distr}$ corresponds to $Y\cap X\times\{0\}$, which is a subvariety.
	
	\paragraph{$\distr'$ is an algebraic (or analytic) subvariety of $TX$.} By \cite{Spallek}, the tangent distribution to $S_{\distr}$ as we defined it in the previous section coincides with the holomorphic Zariski tangent space (which coincides with the algebraic one for an algebraic subvariety), so it is a subvariety of $TX$. Therefore $\distr'=\distr\cap TS_{\distr}$ is an intersection of two subvarieties of $TX$, hence a subvariety.
	
	\paragraph{$\core(\distr)$ is reached (locally) in a finite number of steps.} The ring of germs of polynomials and the ring of germs of holomorphic functions are Noetherian, hence desceding chains stabilize in a finite number of steps. Therefore, around each point of $TX$, the sequence $\distr\supset \distr'\supset\ldots$ stabilizes in a finite number of steps, which implies that, around each point of $TX$, $\core(\distr)$ is a subvariety (as it is obtained, locally, as the intersection of finitely many subvarieties).

As a consequence, the support of the core is a complex algebraic (or analytic) subvariety of $X$.	
	\end{proof}

\begin{ex}Consider a holomorphic map $F:\C^n\to\C^m$ and define
$$\Omega=\{(z',z_{n+1})\in\C^n\times\C\ :\ \Re z_{n+1}<\|F(z')\|^2\}\;.$$
The distribution given by the kernel of the Levi form of $b\Omega$ is diffeomorphic to (by projection) $\ker \mathrm{Jac}F(x)$ as a distribution in $\C^n$; the latter is a complex analytic distribution and, as the construction of the core is invariant by diffeomorphisms, this shows that the core of the distribution given by the kernel of the Levi form of $b\Omega$ is reached in a finite number of steps.
\end{ex}	

We also have the following geometric results about complex algebraic (or analytic) distributions.
\begin{lemma}Let $X$ be a complex algebraic (or analytic) manifold and $\distr\subseteq TX$ be a distribution which is also a complex algebraic (or analytic) subvariety of $TX$; denote by $Y$ the union of the irreducible components of $\distr$ different from $X\times\{0\}$, then for every $(x,v)\in TX$ we have
$$T_{(x,v)}Y=T_xS_{\distr}\times \distr_x\,.$$
\end{lemma}
\begin{proof}Take $x\in S_{\distr}$, otherwise there is nothing to prove. It is obvious that $T_xS_{\distr}\times \distr_x$ is a subspace of $T_{(x,v)}Y$, because $S_{\distr}\times\{0\}\cup\{x\}\times\distr_x$ is contained in $Y$. Now, for the reverse inclusion, consider a complex algebraic (or analytic) manifold $M\subseteq X$ such that $S_{\distr}\subseteq M$ and $T_x S_{\distr}=T_xM$; in a local trivialization of $TX$, consider the product $M\times \distr_x=D$. There are two cases:
\begin{itemize}
\item either $Y$ is locally contained in $D$ around $(x,v)$, then $T_{(x,v)}Y\subseteq T_{(x,v)}M=T_xS_{\distr}\times \distr_x$
\item or, locally around $(x,v)$,  $Y\cap D=\{x\}\times \distr_x$;
\end{itemize}
in the second case, the fiber-wise projection $\pi:TX\vert_M\to D$ given by $\pi(y,v)=(y,\pi_y(v))$ where $\pi_y:T_yX\to \{y\}\times\distr_x$ is a linear projection, is a holomorphic injective map. Hence, locally near $(x,v)$, $\pi\vert_{Y}:Y\to\pi(Y)$ is a biholomorphism with its image; it is clear that $T_{(x,v)}\pi(Y)\subseteq T_{(x,v)}D=T_xS_{\distr}\times \distr_x$ and $D\pi$ (the Jacobian of $\pi$) is the identity along such subspace of $T(TX)$.
\end{proof}

\begin{proposition} If $X$ is a complex algebraic (or analytic) manifold and $\distr\subseteq TX$ is a distribution which is also a complex algebraic (or analytic) subvariety of $TX$, then there exists $\modR\subseteq\Omega^1(X)$ a $\OO$-module of $1$-forms (with $\OO$ the structural sheaf) such that $\distr$ is described as the zero locus of $\modR$.
\end{proposition}
\begin{proof}
Let $\OO_{TX}$ be the structure sheaf of $TX$ as a complex algebraic (or analytic) manifold and take a germ $f\in \OO_{TX,(x,v)}$ such that $f\vert_{\distr}\equiv 0$. In a local trivialization of $TX$, with coordinates $(y,w)$, we can write
$$f=f_0+f_1+f_2+\ldots$$
where each $f_j$ is a homogeneous polynomial of degree $j$ in the variables $w$; then, as $\distr_y$ is a linear subspace of $T_yX$ for every $y\in X$, it is easy to see that $f_0\equiv 0$ and $\distr_y\subseteq \{w\in T_yX\ :\ f_1(y,w)=0\}$.

On the other hand, given $\tilde{w}\not\in \distr_x$, by the previous lemma $W=(0,\tilde{w})\not\in T_x\distr$, therefore there is $f\in \OO_{TX,(x,\tilde{w})}$, vanishing along $\distr$, such that $\langle W, df\rangle\neq 0$ in $(x,\tilde{w})$. But then, as $\distr_x$ is a linear subspace, $\langle W, df\rangle\neq 0$ at $(x,0)$ as well. This implies that $f_1(x,\tilde{w})\neq 0$.

Therefore, for any $f\in \OO_{TX, (x,0)}$ which vanishes identically on $\distr$, consider $f_1$ as defined above and write
$$f_1=\alpha_1(x)v_1+\ldots+\alpha_n(x)v_n\;.$$
To this function we associate the differential form $\alpha=\alpha_1(x)dx_1+\ldots+\alpha_n(x)dx_n\in\Omega^1(X)$; the set of all these differential forms is a $\OO_X$-module, that we denote by $\modR$, such that
$$\{(x,v)\in TX\ :\ \langle \alpha(x),v\rangle=0\ \forall\;\alpha\in\modR\}=\distr\;,$$
by what we proved above.
\end{proof}

In view of these results, we give an algebraic version of the construction of the derived distribution, applied to modules of $1$-forms.
	
Let $X$ be a manifold and $\alg=\alg_X\subseteq\smooth^\infty_X$ be a sheaf of $\R$-algebras (or $\C$-algebras, in case we consider complex valued functions); we furthermore assume that $\mathcal{A}$ gives local coordinates and is closed under differentiation, that is, we require the following two-part property: \begin{enumerate}
		\item[A1)] every $p\in X$ has an open neighborhood $U$ and $x_1,\ldots, x_n\in \mathcal{A}(U)$ such that $(x_1,\ldots, x_n)$ is a system of local coordinates, and 
		\item[A2)] if $f\in \mathcal{A}(U)$, then there exist $f_1,\ldots, f_n\in \mathcal{A}(U)$ such that $df=\sum_{j=1}^nf_jdx_j$. 
	\end{enumerate}
	These properties give a naturally induced sheaf of algebras $\alg_{TX}$ on $TX$.
	
 We denote by $\Omega^\bullet_\mathcal{A}=\bigoplus_{k=0}^n\Omega^k_\mathcal{A}$ the sheaf of smooth differential forms that can be locally represented as a finite sum of terms of the form $fdg_1\wedge \ldots\wedge dg_k$, where $f,g_j$ are local sections of $\mathcal{A}$. It is easily checked that $\Omega^\bullet_\mathcal{A}$ is closed under wedge product and 
	exterior differentiation. 
	
\begin{definition}Given $S\subseteq X$, we define $\ideal_S$ as the ideal of germs $f\in\alg$ which vanish identically on $S$; moreover, we define the \textbf{$\alg$-cotangent module} of $S$ as
the sheaf of $\alg$-modules $T^*_\alg S$ generated by
$$\{df\ :\ f\in\ideal_{S}\}\subseteq\Omega^1_{\alg}\;.$$\end{definition}
\begin{definition}
Given a sheaf of $\alg$-modules $\modR\subseteq \Omega^1_\mathcal{A}$, the \textbf{support} of $\modR$ is 
$$Z_{\modR}=\{x\in X\ :\ \alpha_1(x)\wedge\ldots\wedge\alpha_n(x)=0\ \forall\;(\alpha_1,\ldots,\alpha_n)\in\modR^n\}$$
hence it is the vanishing locus of the sheaf of modules generated by all the $n$-forms $\alpha_1\wedge\ldots\wedge\alpha_n$ where $\alpha_1,\ldots, \alpha_n\in\modR$.

The \textbf{derived module} $\modR'$ is defined as  the sheaf of $\alg$-modules generated by $\modR$ and $T^*_\alg Z_{\modR}$.\end{definition}

\begin{remark}Proceeding by transfinite induction, it is possible to define a concept analogous to the core, a \emph{core module}; however we do not pursue this direction, as in general there is no clear link between this module and the core of the starting distribution.\end{remark}

If $\alg=\smooth^\infty$, or if $X$ is complex algebraic (or analytic) and $\alg$ is the corresponding structure sheaf, we know that, defining $\distr$ as the kernel of $\modR$ (we write $\distr=\ker\modR$), we have that $Z_{\modR}=S_{\distr}$, $\ker T^*_\alg Z_{\modR}=TS_{\distr}$, $\distr'=\ker\modR'$. In such cases, this construction is just an algebraic counterpart of the derived distribution construction; for complex algebraic (or analytic) functions, the algebraic properties of the structure sheaf allow us to derive further conclusions (like the local termination in a finite number of steps and the complex algebraicity - or analyticity - of the core). 

In general, however, we incur in some problems:
\begin{itemize}
\item we do not know that $Z_{\modR}$ is the zero set of an ideal of $\alg$
\item we do not know that $\ker T^*_\alg Z=TZ$ (even if $Z$ is the zero set of an ideal of $\alg$)
\item we do not have any Noetherianity assumption on $\alg$.
\end{itemize}

In the case of real analytic functions, we have the Noetherianity of the local ring; to take care of the other two issues, we consider a \emph{coherent} real analytic space (see Section 1.2 in \cite{AcqBroFer} for details).

\begin{proposition}Let $X$ be a real analytic manifold and $\distr$ be a distribution which is also a real analytic coherent subspace of $TX$, i.e. such that its ideal sheaf $\ideal_{\distr, TX}\subseteq \OO_{TX}$ is \emph{coherent}. Then $S_{\distr}$ is a real analytic coherent subspace of $X$, $\distr'$ and $\core(\distr)$ are again a real analytic coherent subspace of $TX$.\end{proposition}
\begin{proof} By coherence, we can repeat verbatim the proof of Proposition \ref{prp_holalgcore}: $\distr$ admits a decomposition in irreducible components and, if $Z$ is an irreducible real analytic space such that $Z_{(x,v)}=\distr_{(x,v)}$ for some $(x,v)\in TX$, then $Z$ is an irreducible component of $\distr$. Therefore, $S_{\distr}$ is defined as the intersection of real analytic coherent subspaces, hence it belongs to the same category; moreover, the quotient of coherent sheaves is coherent, as it is the dual of a coherent sheaf, so the tangent distribution to a real analytic coherent space is a real analytic coherent subspace of $TX$. Again by \cite{Spallek}, the real analytic tangent distribution coincides with the smooth one and we reach the conclusion for $\distr'$ in the same way. By Noetherianity, we have that the core is reached, locally at each point, in a finite number of steps, hence it is also a real analytic coherent subspace of $TX$, as it is its support.
\end{proof}

\begin{remark}In particular, given a pseudoconvex domain $\Omega$ with real-analytic boundary in some complex manifold $X$, if the kernel of the Levi form gives a coherent distribution in $Tb\Omega$ (viewed as a real-analytic manifold), then the Levi core of $\Omega$ is a coherent real-analytic subset with positive holomorphic dimension. Hence, the results of \cite{DieFor} apply, giving the existence of at least a complex curve inside $b\Omega$ (and hence inside the Levi core).\end{remark}

Without the coherence hypothesis, things break down quite soon.

We first give an example of a non coherent real analytic subspace of $\R^3$, which is nonetheless globally defined by one real analytic function; a coherent sheaf $\mathcal{F}$ is, in particular, of finite type, i.e. at every point $x\in X$ there exist an open set $U\ni x$ and a finite number of sections $s_1,\ldots, s_k\in\Gamma(\mathcal{F},U)$ such that $(s_1)_y,\ldots, (s_k)_y$ generate $\mathcal{F}_y$ for all $y\in U$.

\begin{ex}
		Let $M=\R^3$ and $\mathcal{A}$ be the algebra of real analytic functions; consider $\ideal=(x^2-yz^2)$. Its zero set $W=Z(\ideal)$ is called \emph{Whitney umbrella}. We have that $\ideal_W$ is not of finite type: pick any neighborhood $U$ of $(0,0,0)$, then $\ideal_W(U)\subseteq (x^2-yz^2)$ by real-analyticity, but, for each $p=(0,t,0)$ with $t<0$, $(\ideal_W)_p=(x,z)$.
	\end{ex}
	
	We notice that, in the case of the Whitney's umbrella, the problem is not the requirement for a \emph{finite} number of sections: all the sections of $\ideal_W$ on a neighborhood of $(0,0,0)$ do not generate $\ideal_W$ at any point of the form $(0,t,0)$ with $t<0$.
	
Next, we show that the support of a real analytic distribution may not be a real analytic space, never mentioning it being coherent.
	
\begin{ex}\label{ex:whitney}Consider the Whitney umbrella $W$ and let $\modR$ be the module of $1$-forms given by the cotangent sheaf $\ideal_W/\ideal_W^2$ (as $\ideal_W$ is not of finite type, $\modR$ is also not of finite type). Then $\ker\modR$ is the tangent distribution of $W$, in particular $\ker\modR_p$ with $p=(0,t,0)$, $t<0$, is the distribution locally generated by $\partial_y$; now, set $\distr=\ker\mod\R\cap\ker\mathcal{A}dy$. As $\dim(\ideal_W/\ideal_W^2)^*_p=2$ if $W\ni p\neq (0,t,0)$, $t<0$, we have that $\distr_p\neq 0$ if $p\in W\cap\{y\geq 0\}$; however, $\distr_{(0,t,0)}=\{0\}$ for $t<0$. So $S_{\distr}=W\cap\{y\geq 0\}$ which is not an analytic set in $\R^3$.\end{ex}

Moreover, even if the distribution is defined by a finite type ideal in $TX$, it may happen that its support is not of finite type.	
	
\begin{ex}Consider again $M=\R^3$, $\mathcal{A}$ the algebra of real analytic functions and $\modR$ the module of $1$-forms generated by
		$$xdx-ydz,\ zdx+dy,\ zdy+xdz\;.$$
		If we set $\distr=\ker\modR$, we have that $S_\distr$ is the Whitney umbrella $W$, hence not of finite type, even if $\modR$ is globally generated.\end{ex}
	
	
\subsection{Comments on the general case}
	
An algebraic approach to the reduction to the core finds the following problems, which seem to be interconnected:
	\begin{itemize}
		\item even if the module of $1$-forms defining a distribution is locally generated at each point by its sections, this may not be true for the module defining the derived distribution
		\item even if the module of $1$-forms defining a distribution is locally generated at each point by its sections, this may not be true for the ideal of germs vanishing on its support
		\item if the module of $1$-forms defining a distribution is not locally generated by its sections, its support may not be the zero set of an ideal
		\item even if an ideal $\ideal$ is locally generated at each point by its sections, this may not be true for the radical of such an ideal or for $\ideal(Z(\ideal))$.
	\end{itemize}
	All this issues remain even if we consider a finite number of generating local sections, i.e. sheaves of finite type.
	
	Coherence, in the real-analytic case, helped us to overcome such difficulties; in general, if we can find a good class of sheaves whose elements are locally generated at each point by their sections and which is stable under suitable operations, we may construct a sequence of ideals associated to the derived distributions and hence prove that the core of a distribution is (in such class) a zero set of an ideal.
	
	We notice that, in constructing the module $\modR'$, we had to consider the ideal $\ideal_{Z_{\modR}})$; when we are in $\C^n$ or $\R^n$, $Z_{\modR}$ can be described as the zero set of an ideal $\ideal_\modR$, as all the wedge products $\alpha_1\wedge\ldots\wedge\alpha_n$ will have the form $fdx_1\wedge \ldots\wedge dx_n$. Therefore we are reduced to consider an ideal of the form $\ideal_{Z_{\ideal_\modR}}$.
	
	In many situations, we can describe more precisely this ideal, thanks to Null\- stellen\- satz-type results.
	For example, if we are working with germs of complex polynomials (or holomorphic functions), $\ideal_{Z_{\ideal_\modR}}=\sqrt[h]{\ideal_\modR}$, where $\sqrt[h]{\cdot}$ is the radical appearing in Hilbert's Nullstellensatz (the "usual one"):
	$$\sqrt[h]{\mathcal{J}}=\{f\in\alg :\ g-f^k=0\ \textrm{for some }k\in\N,\ g\in\mathcal{J}\}\;.$$
	In the real (or real analytic) case, the real radical $\sqrt[r]{\cdot}$ plays the same role, where
	$$\sqrt[r]{\mathcal{J}}=\{f\in\alg :\ g-f^{2m}=\sum h_i^2 \;,\ \ m\in\N,\  h_i\in\alg,\  g\in\mathcal{J}\}\;.$$
	Other kinds of radical can be defined, for example, for real germs, adding a \emph{convexity} property (as in \cite{Kohn1979}), we obtain what is called the Lojasiewicz radical in \cite{AcqBroNic},
	$$\sqrt[c]{\mathcal{J}}=\{f\in\alg\ :\ \exists\;g\in\mathcal{J},\ m\in\N\ \textrm{s.t.}\ g-f^{2m}\geq 0\}\;;$$
	this radical appears, for example, in a version of the Nullstellensatz for Denjoy-Carleman quasianalytic functions.
	
	\medskip
	
	It is natural to give the following definition.
	
	\begin{definition}
	Given an $\alg$-module of germs of differential $1$-forms, we can define $\modR^c$
	as the module generated by $\modR$ and
	$$\{df\ :\ f\in\sqrt[c]{\ideal_{\modR})}\}$$
	We define the \textbf{$c$-derivation} as $\distr'_c=\ker\modR^c$. 
	\end{definition}
	\begin{remark}Obviously, we can define $\modR^*$ and $\distr'_*$ with $*\in\{\textrm{h, r}\}$, in an analogous way, however we will only need them for $*=c$ in what follows.\end{remark}
	
	\begin{remark}If a Nullstellensatz-type result holds, one can show the equality between $\modR'$ and the corresponding $\modR^*$ and also between the related cores.\end{remark}
	
	\section{Kohn's algorithm}
	
	We discuss in this section the relation between the construction of the derived distributions and Kohn's algorithm; we refer to \cite{Kohn1979} for the meaning of Kohn's algorithm in relation to the subellipticity problem for $\dbar$ and $\dbar$-Neumann operators.
	We recall the definition of Kohn's algorithm of multiplier ideal sheaves for $(0,1)$-forms: given $\Omega$ a domain in $\C^n$, let $x_0\in b\Omega$ and take a germ $r\in\mathcal{C}^\infty_{x_0}$ such that $r\equiv 0$ on $b\Omega$, but $dr$ does not vanish; define
	$$I_0(x_0)=(r,\star \partial r\wedge\dbar r\wedge (\de\dbar r)^{n-1})$$
	$$I_k(x_0)=\sqrt[c]{(I_{k-1}(x_0), A_{k-1}(x_0))}$$
	where $\star$ is the Hodge $\star$-operator and 
	$$A_{k-1}(x_0)=\big\{\star \de f_1\wedge\dbar f_1\wedge\ldots\wedge \de f_j\wedge\dbar f_j\wedge\partial r\wedge\dbar r\wedge(\de\dbar r)^{n-1-j}\ :$$
	$$\phantom{.}\qquad\qquad\ j\in\N, f_1,\ldots, f_j\in I_{k-1}(x_0)\big\}\;.$$
	
	\begin{theorem}\label{corekohn}If we consider the Levi-null distribution $\levinull$ on $b\Omega$ and its sequence f derived distributions $\levinull=\distr^{(0)}\supseteq \distr^{(1)}\supseteq\ldots\supseteq \core(\levinull)$, we have that
		$$I_k(x_0)\subseteq \ideal(Z(I_{\distr^{(k)}, x_0}))$$
		for all $k\in \N$, where $\ideal_{\distr^{(k)}}$ is the ideal sheaf of the functions that vanish identically on $S_{\distr^{(k)}}$ (whose zero set may, in principle, be larger than the support of $\distr^{(k)}$).
		
		On the other hand, let $\distr^{(\alpha)}_c$ be the sequence of null distributions given by the c-derivation $\modR\mapsto\modR^c$, starting from $\levinull$; then
		$$I_k(x_0)= \sqrt[c]{\ideal_{\distr^{(k)}_c, x_0}}$$
		for all $k>0$.
	\end{theorem}
	\begin{proof}Let $z_1,\ldots, z_n$ local coordinates around $x_0$, then the module of $1$-forms $\modR$ generated by $\partial r,\dbar r$ and the collections
		$$\de\dbar r(\cdot,\partial_{\bar{z}_j}\ \textrm{ for } j=1,\ldots, m\qquad rdz_j\ \textrm{ for }j=1,\ldots, m\;.$$ 
		Then $\levinull=\ker\modR$ and $I_0(x_0)=\ideal_{\levinull,x_0}$.
		
		Now, $\modR^{(1)}$ is generated by $\modR$ and by the set
		$$\{df\ :\ f\in\ideal(Z(\ideal_{\levinull}))\}\;.$$
		Therefore $\modR^{(1)}_{x_0}$ contains $\modR_{x_0}$ and $\{df\ :\ f\in I_{0}(x_0)\}$; as $\distr^{(1)}=\ker\modR^{(1)}$, when we compute $\ideal_{\distr^{(1)}}$, by definition we consider all the determinants of the matrices obtained using as rows the coefficients of elements of $\modR^{(1)}$. It is clear that this includes all the expressions of the form
		$$\star \de f_1\wedge\dbar f_1\wedge\ldots\wedge \de f_j\wedge\dbar f_j\wedge\partial r\wedge\dbar r\wedge(\de\dbar r)^{n-1-j}$$
		with $j\in\N, f_1,\ldots, f_j\in I_{0}(x_0)$. Therefore
		$$(I_0(x_0), A_0(x_0))\subseteq \ideal_{\distr^{(1)},x_0}\;,$$
		which implies the thesis for $k=1$.
		
		To iterate the argument it is enough to recall that $\sqrt[c]{\ideal}\subseteq \ideal(Z(\ideal))$.
	\end{proof}
	
	\begin{remark}\begin{enumerate}
			\item If the defining function $r$ can be chosen in a suitable subring $\alg$ of $\mathcal{C}^\infty$, then the ideals $I_k(x_0)$ can be considered as ideals in $\alg$.
			\item In some cases, for example when $\alg$ is the ring of germs of holomorphic or real analytic functions, we know that $\sqrt[c]{\ideal}=\ideal(Z(\ideal))$ and then the reduction to the core coincides with Kohn's algorithm.
			\item We recall that in the particular cases mentioned in the previous point, the ring $\alg$ is Noetherian, hence, around each point, the core is reached in a finite number of steps. 
	\end{enumerate}\end{remark}
	
	We would like to give an example of a non Noetherian ring of functions where nonetheless a Nullstellensatz-type result holds, allowing us to conclude, as in the previous remark, that the reduction to the core and Kohn's algorithm are two instantiations of the same idea.
	
	\begin{definition}Let $\mathcal{C}$ be a sheaf such that, for every open set $U$,
		\begin{enumerate} 
			\item $\mathcal{C}(U)$ is a $\R$-subalgebra of $\mathcal{C}^\infty(U)$.
			\item $\mathcal{C}^\omega(U)\subseteq\mathcal{C}(U)$
			\item $\mathcal{C}$ is closed under composition with mappings whose components are in $\mathcal{C}$
			\item $\smooth$ is closed under differentiation
			\item $\smooth$ is quasi-analytic, i.e., if $f\in\smooth(U)$, $a\in U$ and the Taylor series of $f$ at $a$ is identically zero, then $f\equiv 0$ in a neighborhood of $a$
			\item $\smooth$ is closed under division by a coordinate
			\item $\smooth$ is closed under inverse and hence satisfies the Implicit Function Theorem.
		\end{enumerate}
	\end{definition}
	
	An example of such a $\smooth$ is a Denjoy-Carleman quasianalytic class (see \cite{BieMil} for the basic definitions, \cite{Rudin} for the classical theory and \cite{Thill} for a more general approach).
	
	It follows from \cite{BieMil} that the resolution of singularities holds for finitely generated ideals in $\smooth$; in the same paper (Theorems 6.1 and 6.3), the authors also prove what is known as \emph{topological (or geometric) Noetherianity} and \emph{Lojasiewicz inequalities}.
	\begin{theorem}[Topological Noetherianity] \label{topnoet} A decreasing sequence of germs of $\smooth$-sets, stabilizes in a finite number of steps.\end{theorem}
	
	\begin{theorem}[Lojasiewicz inequalities]Let $f,g\in\mathcal{C}_p$ and suppose that $\{x\ :\ g(x)=0\}\subseteq\{x:f(x)=0\}$ as germs in $p$, then there exist $c, \lambda>0$ such that
		$$|g(x)|\geq c|f(x)|^\lambda\;.$$
		In general, if $Z=\{x:f(x)=0\}$, we can find $c,\nu>0$ such that 
		$$|f(x)|\geq c\mathrm{dist}(x,Z)^\nu\;.$$
	\end{theorem}
	
	Localizing the argument from \cite{AcqBroNic}[Theorem 1.1 - part (i)] and employing the two results above, we obtain the following Nullstellensatz-type theorem for the class $\smooth$.
	\begin{theorem}
		If $\ideal\subseteq\smooth$ is a finitely generated ideal, then $$\sqrt[c]{\ideal}=\ideal(Z(\ideal))$$.
	\end{theorem}
	
	Moreover, as a consequence of Theorem \ref{topnoet}, we obtain the following.
	
	\begin{proposition}\label{prop_fingen}Let $\ideal\subset\smooth_p$ be any ideal of germs at the point $p$. Then there exists $\mathcal{J}\subseteq\ideal$, finitely generated, such that $Z(\mathcal{J})=Z(\ideal)$.
	\end{proposition}
	\begin{proof} If $\ideal$ is the zero ideal, it is finitely generated. If not, there is $g_1\in\ideal$ such that $(g_1)\neq 0$; if $Z((g_1))=Z(\ideal)$, we set $\mathcal{J}=(g_1)$, otherwise we find $g_2\in\ideal$ such that $Z((g_1,g_2))\neq Z((g_1))$. If $Z((g_1,g_2))=Z(\ideal)$, we set $\mathcal{J}=(g_1,g_2)$, otherwise we continue as before.
		
		If we never obtain that $Z((g_1,\ldots, g_k))=Z(\ideal)$, we produce a sequence of ideals
		$$(g_1)\subseteq (g_1,g_2)\subseteq (g_1,g_2,g_3)\subseteq \ldots\;.$$
		By Theorem \ref{topnoet}, the sequence of the zero loci stabilizes after a finite number of steps; this means that
		$$Z((g_1,\ldots, g_k))=Z((g_1,\ldots, g_k, g_{k+1}))$$
		which contradicts the construction of the sequence $\{g_k\}$. Therefore the previous construction will stop in a finite number of steps, giving
		$$Z(\mathcal{J})=Z(\ideal)$$
		with $\mathcal{J}=(g_1,\ldots, g_k)$.
	\end{proof}
	
	Combining these two results and recalling the remark after Theorem \ref{corekohn}, we obtain that, if $r\in\smooth$, then $\distr^{(k)}_c=\distr^{(k)}$ for all $k\in\N$; therefore
	$$I_k(x_0)\cap\smooth_{x_0}=I_{\distr^{(k)},x_0}\cap \smooth_{x_0}\;.$$
	As a consequence of Theorem \ref{topnoet}, the reduction to the core and Kohn's algorithm terminate in a finite number of steps; moreover, Kohn's algorithm gives a subelliptic estimate if and only if the core is trivial.
	
	\begin{remark} We notice that the results we employed here are also presented in \cite{Nicoara}, where they are used to prove that Kohn's algorithm terminates succesfully if and only if the boundary is of finite D'Angelo type for domains defined by a function in a Denjoy-Carleman quasianalytic class.\end{remark}
	
	\section{Other examples of cores}
	
	The definition of core can be specialized in a number of interesting situations. 
	
	\begin{enumerate}
		\item Given a complex manifold $M$ and a plurisubharmonic function $\phi:M\to\R$, we define the distribution $\levinull_\phi$ given by the kernel of $\de\dbar \phi$ and its core $\core(\levinull_\phi)$ will be referred to as the \emph{core of $\phi$}.
		\item Given a weakly complete complex manifold $M$, let $\mathcal{E}$ be the set of (smooth) plurisubharmonic exhaustion functions; we define 
		$$\levinull_{\mathcal{E}}=\bigcap_{\phi\in E}\levinull_\phi$$
		and the \emph{weakly complete core} as $\core(\levinull_{\mathcal{E}})$.
		\item Given a weakly complete complex manifold $M$, let $\mathcal{E}$ be the set of (smooth) plurisubharmonic exhaustion functions; we define 
		$$\distr_{\mathcal{E}}=\ker\{d\phi\}_{\phi\in E}$$
		and the \emph{psh core} as $\core(\distr_{\mathcal{E}})$.
	\end{enumerate}
	
	\begin{remark}Given a domain in $\C^{n+1}$ of the form
		$$\{(z',z_n)\in\C^n\times\C\ :\ \Re z_n<\phi(z')\}$$
		for $\phi:\C^n\to\R$, there is a natural identification between $S_{\core(\levinull_\phi)}$ and $S_{\core(\levinull)}$.\end{remark}
	
	\begin{proposition}\label{kernels}We have that $S_{\levinull_{\mathcal{E}}}=S_{\levinull'_{\mathcal{E}}}=S_{\distr_{\mathcal{E}}}=S_{\distr'_{\mathcal{E}}}$; therefore, the derived distributions $\levinull'_{\mathcal{E}}$ and $\distr'_{\mathcal{E}}$ are \emph{perfect}, hence they coincide with their core.\end{proposition}
	\begin{proof}
		From \cite[Proposition 4.2]{BiaMon}, we have that $S_{\levinull_{\mathcal{E}}}=S_{\distr_{\mathcal{E}}}=\Sigma_M$ is the \emph{minimal kernel} of $M$, as defined in \cite{SloTom}; therefore, by \cite[Lemma 3.1]{SloTom}, there exists a \emph{minimal function}, i.e. a smooth plurisubharmonic exhaustion function $\phi:M\to\R$ such that $\phi$ is strictly plurisubharmonic outside of $\Sigma_M$.
		
		By repeated applications of \cite[Lemma 3.4]{MonSloTom2}, given $\phi_1,\ldots, \phi_k\in\mathcal{E}$, and real numbers $c_0$,$\ldots$, $c_k$, we have that
		$$\Sigma_M\cap\{\phi=c_0,\ \phi_1=c_1,\ldots,\ \phi_k=c_k\}$$
		is a local maximum set, if not empty. We need the following property of local maximum sets: if $K$ is a local maximum set, then $T_xK$ contains at least a complex line for each $x\in K$. We prove it separately as Lemma \ref{complextg}, after the end of this proof.
		
		Given $p\in S_{\distr_{\mathcal{E}}}$, we have that there exist $\phi_1,\ldots, \phi_k\in\mathcal{E}$ such that 
		$$\distr_{\mathcal{E},p}=(\ker d\phi_1\cap\ldots\cap\ker d\phi_k)\vert_p\;.$$
		Define
		$$F=S_{\distr_{\mathcal{E}}}\cap\{x\in M\ :\ \phi(x)=\phi(p),\ \phi_1(x)=\phi_1(p),\ \ldots,\ \phi_k(x)=\phi_k(p)\}\;;$$
		as we noted before, $F$ is a local maximum set (it contains at least $p$, so it is not empty) and $T_pF$ contains at least a copy of $\C$. Note that
		$$T_pF\subseteq T_pS_{\distr_{\mathcal{E}}}\cap (\ker d\phi_1\cap\ldots\cap\ker d\phi_k)\vert_p=T_pS_{\distr_{\mathcal{E}}}\cap \distr_{\mathcal{E},p}=\distr'_{\mathcal{E},p}\;.$$
		Therefore $p\in S_{distr'_{\mathcal{E}}}$.
		
		\medskip
		
		Let now $p\in S_{\levinull_{\mathcal{E}}}$. By \cite[Theorem 5.2]{Slod18}, 
		$$S_{\levinull_{\mathcal{E}}}=\Sigma_M=\bigcup_{\alpha\in A}F_\alpha$$
		where each $F_\alpha$ is a compact local maximum set on which every $\phi\in\mathcal{E}$ is constant. 
		
		By adapting the proof of \cite[Theorem 3.13]{Levicore}, we can show that, for every $\phi\in\mathcal{E}$, $\ker\de\debar\phi\cap TF_\alpha\cap JTF_\alpha$ has support equal to $F_\alpha$.
		
		Therefore, the support of $TF_\alpha\cap \ker \de\debar\phi\subseteq TS_{\levinull_{\mathcal{E}}}\cap\ker\de\debar\phi$ contains $F_\alpha$ for all $\alpha\in A$ and all $\phi\in\mathcal{E}$, hence the support of $\levinull'_{\mathcal{E}}=TS_{\levinull_{\mathcal{E}}}\cap\levinull_{\mathcal{E}}$ contains $F_\alpha$ for all $\alpha\in A$, i.e. coincides with the support of $\levinull_{\mathcal{E}}$.
	\end{proof}
	
	\begin{lemma}	\label{complextg}If $K\subset M$ is a local maximum set in a complex manifold, then $T_xK$ contains at least a complex line, for all $x\in \overline{K}$.\end{lemma}
	\begin{proof}If $\dim_\R T_xK>\dim_\C M$, then obviously $T_xK\cap JT_xK\neq\{0\}$ by Grassmann's formula. So, let us suppose that $k=\dim_\R T_xK\leq n$.
		
		Now, let $(V,x)$ be a germ of $k$-dimensional real submanifold such that $(K,x)\subseteq (V,x)$ and $T_xK=T_xV$; if $T_xK$ is totally real, then $T_yV$ is totally real for all $y$ in a neighborhood of $x$ in $V$. As $(V,x)$ is a germ of totally real submanifold, we can find a psh function $u$, defined on a small neighborhood $U\subseteq X$ of $x$, such that $u(x)=0$ and $u(y)<0$ for all $y\in U\cap V$. But, as $(K,x)\subseteq (V,x)$, up to shrinking $U$, we can suppose that $K\cap U\subseteq V\cap U$, so that $u$ would violate the local maximum property of $K$.
		
		Therefore, $T_xK$ cannot be totally real, so it contains at least a complex line.
	\end{proof}
	
	The support of $\levinull_{\mathcal{E}}$ is the so called minimal kernel of a weakly complete space; its geometry has been studied in a number of papers by the second author and others (see e.g. \cite{SloTom,MonSloTom1, MonSloTom2, MonSloTom3, MonTom, Slod18} and the references therein). This ``weak integrability property" given by Proposition \ref{kernels} surely hints at the existence of some analytic structure, which, up to now, has been found only in the case of complex surfaces.

	\section{Rigid domains}
In this section we discuss the relation between Catlin's property (P) and the triviality of the core (see also \cite{Treuer}).
	
	Let $h:\C^n\to\R$ be a $\smooth^2$ function and define
	$$\Omega_h=\{(z',z_{n+1})\in\C^n\times\C\ :\ \Re(z_{n+1})+h(z')<0\}\;.$$
	Let $X=b\Omega_h$, then for $p=(p',p_{n+1})\in X$, 
	$$T_p^{1,0}X=\{Y'_p+\gamma_p\de_{z_{n+1}}\ :\ Y'_p\in T^{1,0}_{p'}\C^n,\ 2\de h (Y'_p)+\gamma_p=0\}\;.$$
	
	In particular, the map $T^{1,0}_pX\ni (Y'_p+\gamma\de_{z_{n+1}})\mapsto Y'_p\in T^{1,0}_{p'}\C^n$ is a $\C$-linear isomorphism, which gives an isomorphism of complex vector bundles from $T^{1,0}X$ to $T^{1,0}\C^n$.
	
	Moreover, the Levi form on $X$ is 
	$$\lambda_p(Y'_p+\gamma_p\partial_{z_{n+1}}, \overline{Z}'_p+\overline{\delta}_p\de_{\bar{z}_{n+1}})=i\de\dbar h_p (Y'_p, \overline{Z}'_p)$$
	so it corresponds, under said isomorphism, to the complex Hessian of $h$ on $T^{1,0}\C^n$. Hence, $\Omega_h$ is pseudoconvex if and only if $h$ is plurisubharmonic.
	
	Moreover, $\core(\levinull)$ in $T^{1,0}X$ corresponds, under such an isomorphism to $\core(\levinull_h)$ in $T^{1,0}C^n$.
	
	As a special case, when $n=1$, we have that $h:\C\to\R$ should be subharmonic; given $K\subseteq \C$ closed, we can find a $\smooth^\infty$ function $g:\C\to[0,+\infty)$ such that $K=\{z\in\C\ :\ g(z)=0\}$.
	
	Solving the equation $\triangle h=g$, we obtain a subharmonic function whose Laplacian vanishes exactly on $K$, so that $\levinull_h=T^{1,0}\C\vert_K$.
	
	At any point $p\in K$, $\levinull_{h,p}\cap \C T_p K$ can be either $\levinull_{h,p}$ or $\{0\}$, depending whether $T_pK$ is $\C$ or not (respectively).
	
	\begin{ex}Let $C\subset\R$ be a Cantor set in $[0,1]$ and define $K=C+iC\subseteq\C$; we consider the function $h$ described above and the corresponding rigid domain $\Omega_h\subseteq\C^2$.
		
		As the Cantor set is a perfect set, $T_xC=\R$ for all $x\in\R$, so that $T_zK=\C$ for all $z\in K$; hence $\levinull_h$ is a perfect distribution in $T^{1,0}\C$ and so is the Levi null distribution $\levinull$ on $b\Omega_h$.
		
		Therefore, the Levi core of $b\Omega_h$ is non trivial and supported on the set
		$$\{(z,w)\in b\Omega_h\ :\ z\in K\}\cong K\times\R\;.$$
		
		We notice that the core here is "small", in the sense that, as the Lebesgue measure of $C$ in $\R$ is zero, the $3$-dimensional Hausdorff measure of $S_{\core(\levinull)}$ is zero inside $b\Omega_h$.
	\end{ex}

	A bounded version of $\Omega_h$ is easily obtained. Let
	$$\Omega'_h=\{(z',z_{n+1})\in\C^n\times\C\ :\ \Re(z_{n+1})+h(z')+ \lambda(\|z'\|^2+|z_{n+1}|^2)<0\}\;,$$
	where $\lambda$ is a non negative, increasing, convex function such that $\lambda(t)=0$ if $t\leq 2$ and $\lambda(t),\lambda'(t),\lambda''(t)>0$ if $t>2$.
	We have that $\Omega_h'$ is pseudoconvex and bounded and strictly pseudoconvex if $\|z'\|^2+|z_{n+1}|^2>2$; moreover, $b\Omega_h\cap B_{\sqrt{2}}(0,0)\subseteq b\Omega_h'$; so,  the support of the core of $\Omega_h'$ is contained in
	$$\{(z,w)\in b\Omega_h\ :\ z\in K\}\cap \cap B_{\sqrt{2}}(0,0)\;.$$
	
	\begin{ex}
		Let $C\subset[0,1]$ be a Cantor set such that
		$$C=\left(\frac{1}{5}C + 0 \right)\cup\left(\frac{1}{5}C + \frac{4}{5}\right)$$
		then $\dim_H C=\log(2)/\log(5)<1/2$ ($\dim_H$ is the Hausdorff dimension). This means that $\dim_H(C\times C\times \R)<2$.
		
		Constructing $\Omega_h'$ on the set $K=C+iC$, we obtain that the support of the core has vanishing $2$-dimensional Hausdorff measure; by \cite[Remarque p.310]{Sibony}, this implies that $U_{\phi,\epsilon}$ is B-regular, or, in an equivalent formulation, that it satisfies Catlin's Property (P).
	\end{ex}
	
	Therefore, we showed that Catlin's property (P) (or B-regularity) does not imply trivial core.
	
	\medskip
	
	Unfortunately, these examples all have trivial D'Angelo class: the complex Hessian of the defining function vanishes completely at the points of the core, so every D'Angelo form is zero. Hence, the Diederich-Fornaess exponent of the domains from the two previous examples is $1$.
	
	\medskip
	
	If we assume some more properties of the set $K$, we can produce other kinds of examples. Suppose that $K$ is a compact made up of points which are regular for the Dirichlet problem, then there exists a smooth function $u:\C\setminus K \to (0,+\infty)$ such that $u(x)\to 0^+$ when $\C\setminus K\ni x\to K$, $u$ has a logarithmic pole at $\infty$, and $\triangle u=0$ on $\C\setminus K$, i.e. $u$ is the Green function with pole at $\infty$ for $\mathbb{CP}^1\setminus K$.
	
	Let $\lambda:[0,+\infty)\to\R$ be a non negative, convex, increasing function such that $\lambda(t)$ vanishes so quickly for $t\to0^+$ such that $\lambda\circ u$ extends to $0$ smoothly on $K$. Let
	$$\Omega_{K,c}=\{(z,w)\in\C^2\ :\ \lambda(u(z))+|w|^2<c\}\;,$$
	which will be, for almost all $c$, a smoothly bounded domain. We have that
	$$\de\debar \lambda\circ u(z)=\left|\frac{\partial u}{\partial z}\right|^2(\lambda''(u)+\lambda'(u))dz\wedge d\bar{z}$$
	so that the complex Hessian of the defining function of $\Omega_{K,c}$ is
	$$\left|\frac{\partial u}{\partial z}\right|^2\lambda''(u)dz\wedge d\bar{z}+dw\wedge d\bar{w}\;.$$
	A vector in $T^{1,0}b\Omega_{K,c}$ is of the form
	$$\bar{w}\partial_z - (\partial_z u)\lambda'(u)\partial_w$$
	hence the Levi form of $b\Omega_{K,c}$ is given by
	$$\left|\frac{\partial u}{\partial z}\right|^2(|w|^2\lambda''(u)+\lambda'(u))$$
	which will vanish if and only if $z\in K$ or $\partial_z u=0$.
	
	Points where $\partial_z u=0$ are of finite type; if $K$ is such that $T_zK=\C$ for all $z\in K$, then $S_{\core(\levinull)}\cong K\times\mathbb{S}^1$.
	
	\begin{ex}If $C\subseteq[0,1]$ is a suitably constructed Cantor set such that $K=C+iC$ has Hausdorff dimension $1$ but $\mathcal{H}^1$-measure $0$, then the Levi core of $b\Omega_{K,c}$ is supported in a set of vanishing $2$-dimensional Hausdorff measure (but of Hausdorff dimension $2$).\end{ex}
	
	We note that, for $(z,w)\in S_{\core(\levinull)}$, a tangent $(1,0)$-vector in $\levinull$ has the form
	$$\bar{w}\partial_z$$
	whereas a $(0,1)$-normal vector will be a multiple of
	$$\bar{w}\partial_w$$
	therefore, as the complex Hessian of the defining function is diagonal, we once again get that the D'Angelo class is trivial (as it contains the null $1$-form).

\end{document}